\documentclass[12pt]{article}
\usepackage{paralist,amsmath,amsthm,amssymb,mathtools,url,graphicx,pdfpages,tikz,pdfpages,rotating}
\usetikzlibrary{positioning,chains,fit,shapes,calc}
\usetikzlibrary{arrows,automata}
\usepackage[affil-it]{authblk}
\usepackage[left=1in,right=1in,top=1.2in, bottom=1in]{geometry}
\usepackage{mathdots}
\usepackage{bm} 

\newtheorem{theorem}{Theorem}[section]
\newtheorem{lemma}[theorem]{Lemma}

\theoremstyle{definition}
\newtheorem{definition}[theorem]{Definition}
\newtheorem{example}[theorem]{Example}
\newtheorem{obs}[theorem]{Observation}

\numberwithin{equation}{section}

\DeclareMathOperator{\n}{N }

\DeclareMathOperator{\von}{von}

\newcommand{\Mod}[1]{ \left(\mathrm{mod}\ #1\right)}

\newcommand{\F}{\mathbb{F}}

\newcommand{\cupdot}{\mathbin{\mathaccent\cdot\cup}}

\begin{document}
\title{A New Formula for the Minimum Distance of an Expander Code}

\author{Sudipta Mallik} 
\affil{\small Department of Mathematics and Statistics, Northern Arizona University, 801 S. Osborne Dr.\\ PO Box: 5717, Flagstaff, AZ 86011, USA 

sudipta.mallik@nau.edu}

\maketitle
\begin{abstract}
An expander code is a binary linear code whose parity-check matrix is the bi-adjacency matrix of a bipartite expander graph. We provide a new formula for the minimum distance of such codes. We also provide a new proof of the result that  $2(1-\varepsilon) \gamma n$ is a lower bound of the minimum distance of the expander code given by a $(m,n,d,\gamma,1-\varepsilon)$ expander bipartite graph.
\end{abstract}


\section{Introduction}

Binary linear codes can be constructed from graphs. One such construction was given from bipartite graphs by Tanner in \cite{Tanner}. Sipser and Spielman constructed expander codes from bipartite expander graphs in \cite{Spielman}. One of the goals of all these constructions was to have linear codes with relatively large minimum distance for efficient error correction. For more details on the literature of linear codes and bipartite graphs, see \cite{alon,Capalbo,Spielman,Zemor}. In this article we provide a new formula for the minimum distance of expander codes. We also provide a new proof of the result that  $2(1-\varepsilon) \gamma n$ is a lower bound of the minimum distance of the expander code given by a $(m,n,d,\gamma,1-\varepsilon)$ expander bipartite graph.

Now we present a brief introduction to coding theory:  A binary linear code $C$ of length $n$ and dimension $k$ is a $k$ dimensional subspace of $\F_2^n$ where $\F_2$ is the binary field. The code $C$ is called an $[n,k]$-code. The support of a codeword $\bm x\in C$ is the set of indices $i$ such that $i$th entry of $\bm x$ is $1$. The {\it Hamming weight} $w_H(\bm x)$ of a vector $\bm x \in \F_2^n$ is the size of the support of $\bm x$. The {\it  Hamming distance}, denoted by $d_H(\bm x,\bm y)$, between two codewords $\bm x$ and $\bm y$ in $C$ is $d_H(\bm x,\bm y) = w_H(\bm x-\bm y)$. The {\it minimum distance} of $C$, denoted by $d(C)$, is the minimum distance between distinct codewords in $C$. Note that $d(C)$ is the minimum Hamming weight of a nonzero codeword in $C$.  We call $C$ to be an $[n,k,d]$ code when $d(C)=d$. A binary matrix $H$ is called the {\it parity-check matrix} of $C$ if $C$ the null space of $H$, i.e., 
$$C = \{\bm c \in \F_2^n | H\bm c^{T} = 0\}.$$

The minimum distance $d(C)$ can be expressed as the minimum number of linear dependent columns of the parity-check matrix of $C$ as follows:

\begin{theorem}\cite[Theorem 2.2]{roth}\label{minmallyDependentColumns}
Let $C$ be a linear code and $H$ its  parity-check matrix. Then $C$ has minimum distance $d$ if and only if any $d-1$ columns of $H$ are linearly independent and some $d$ columns of $H$ are linearly dependent.
\end{theorem}

For a vertex $v$ of a graph $G$, the set of all vertices in $G$ adjacent to $v$ is called the {\it neighbor} of $v$, denoted by $\n(v)$. For a set $S$ of vertices of $G$, $\n(S)$ denotes the union of neighbors of vertices in $S$. Now we define a bipartite expander graph based on its definition in \cite{Spielman,Tanner} with the roles of left and right set of vertices switched:
\begin{definition}
Suppose $G$ is a bipartite graph with vertex set $L\cupdot R$ such that $|L|=m$, $|R|=n$, each edge of $G$ joins a vertex of $L$ with a vertex of $R$, and  each vertex of $R$ is adjacent to exactly $d$ vertices of $L$. For positive $\gamma$ and $\alpha$, $G$ is called a $(m,n,d,\gamma,\alpha)$ {\it expander graph} if for each set $S\subseteq R$ satisfying $|S|\leq \gamma n$, we have $$|N(S)|\geq d\alpha |S|.$$
\end{definition}

\begin{figure}
\centering
\begin{tikzpicture}[scale=0.6, transform shape, thick,
  every node/.style={draw,circle},
  fsnode/.style={fill=blue},
  ssnode/.style={fill=teal},
  every fit/.style={ellipse,draw,inner sep=-2pt,text width=2cm},
  -,shorten >= 0pt,shorten <= 0pt
]

\begin{scope}[start chain=going below,node distance=7mm]
\foreach \i in {1,...,5}
  \node[fsnode,on chain] (f\i) [label=left: \i] {};
\end{scope}

\begin{scope}[xshift=7cm,yshift=-1.1cm,start chain=going below,node distance=7mm]
\foreach \i in {6,...,9}
  \node[ssnode,on chain] (s\i) [label=right: \i] {};
\end{scope}

\node [blue,fit=(f1) (f5),label=below: $L$] {};
\node [teal,fit=(s6) (s9),label=below: $R$] {};

\draw (s6) -- (f1);
\draw (s6) -- (f2);

\draw (s7) -- (f4);
\draw (s7) -- (f5);

\draw (s8) -- (f2);
\draw (s8) -- (f3);

\draw (s9) -- (f1);
\draw (s9) -- (f3);
\end{tikzpicture}

\caption{A $(5,4,2,\frac{1}{2},\frac{2}{3})$ expander graph.}\label{fig1}
\end{figure}
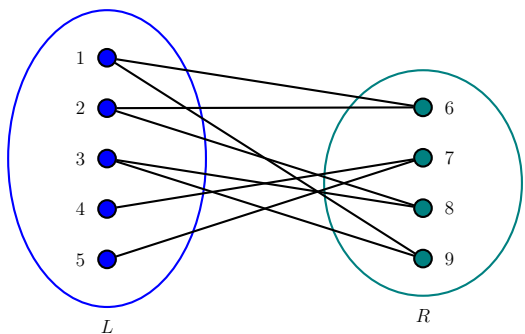

\begin{example}
The bipartite graph in Figure \ref{fig1} is a $(5,4,2,\frac{1}{2},\frac{2}{3})$ expander graph. Each vertex in $R$ has degree $d=2$.  If $S\subseteq R$ satisfies $|S|\leq \gamma n=2$, then $|S|=1$ or $2$. For $|S|=1$,  $|N(S)|=2\geq \frac{4}{3} =d\alpha |S|$. Also for $|S|=2$,  $|N(S)|\geq \frac{8}{3}=d\alpha |S|$.
\end{example}

\begin{definition}
Suppose $G$ is a $(m,n,d,\gamma,\alpha)$ expander graph and $B$ is the $m\times n$ bi-adjacency matrix of $G$, i.e., 
$$A=\left[ \begin{array}{c|c}
O_m&B\\
\hline
B^T& O_n
\end{array}\right]$$
is the adjacency matrix of $G$. The binary linear code whose parity-check matrix is $B$ is called the {\it expander code} of $G$, denoted by $C(G)$. In other words,
$$C(G) = \{\bm c \in \F_2^n \;|\; B\bm c^{T} = \bm 0 \text{ in } \F_2\}.$$
\end{definition}

\begin{example}
The bi-adjacency matrix of the $(5,4,2,\frac{1}{2},\frac{2}{3})$ expander graph $G$ in Figure \ref{fig1} is given by
$$B=\left[\begin{array}{cccc}
     1  &   0   &   0   &   1  \\
     1  &   0   &   1   &   0  \\
     0  &   0   &   1   &   1  \\
     0  &   1   &   0   &   0  \\
     0  &   1   &   0   &   0  
\end{array}\right].$$

The expander code $C(G)$ of $G$ is given by
$$C(G) = \{\bm c \in \F_2^4 \;|\; B\bm c^{T} = \bm 0 \text{ in } \F_2\}.$$
\end{example}

\section{Main Results}
We start with the following notation and definition of $von$ from \cite{Mallik}.

\begin{definition}
For a nonempty subset $S$ of vertices of a graph $G$, the set of vertices of $G$ with odd number of neighbors in $S$ is denoted by $\von(S)$, i.e.,
$$\von(S)=\{v\in V(G)\;:\: |\n(v)\cap S| \text{ is odd}\}.$$
\end{definition}

\begin{example}
Consider the $(5,4,2,\frac{1}{2},\frac{2}{3})$ expander graph $G$ in Figure \ref{fig1}. For $v=6,7,8,9$, $\von(\{v\})=\n(v)$. For $S=\{6,7,8\}$, $\von(S)= \{1,3,4,5\}$. For $S=\{6,8,9\}$, $\von(S)= \varnothing$.
\end{example}

Now we proceed to the main results of this article which give a new formula of the minimum distance of an expander code.

\begin{theorem}\label{mindist1}
Suppose $G$ is a bipartite graph with vertex set $L \cupdot R$ such that $|L|=m$, $|R|=n$, and $B$ is the $m\times n$ bi-adjacency matrix of $G$. Let $S$ be a nonempty subset of $R$. If $\von(S)=\varnothing$, then the columns of $B$ indexed by $S$ are linearly dependent. Conversely if the columns of $B$ indexed by $S$ are minimally linearly dependent, then $\von(S)=\varnothing$.
\end{theorem}
\begin{proof}
Suppose $\von(S)=\varnothing$ where $S=\{i_1,i_2,\ldots,i_t\}\subseteq R$. Then 
$$B_{i_1}+B_{i_2}+\cdots+B_{i_t} \equiv 0 \Mod{2}$$
which implies columns $B_{i_1},B_{i_2},\ldots,B_{i_t}$ of $B$ are linearly dependent.

Conversely suppose $S=\{1,2,\ldots,k\}\subseteq R$ and $B_1,B_2,\ldots,B_k$ are minimally linearly dependent columns of $B$. Then  $B_1+B_2+\cdots+B_k \equiv 0 \Mod{2}$. We claim $\von(S)=\varnothing$. Otherwise let $i\in \von(S)$. Then 
$$(B_1+B_2+\cdots+B_k)_i \equiv 1 \Mod{2},$$
a contradiction to $B_1+B_2+\cdots+B_k \equiv 0 \Mod{2}$. Thus $\von(S)=\varnothing$.
\end{proof}

\begin{theorem}\label{mindist}
Suppose $G$ is a bipartite graph with vertex set $L \cupdot R$ such that $|L|=m$, $|R|=n$, and $B$ is the $m\times n$ bi-adjacency matrix of $G$. Suppose $C$ is the binary linear code whose parity-check matrix is $B$. Then the minimum distance $d(C)$ of $C$ is given by
$$d(C) = \min \{ |S|\;:\; \varnothing \neq S\subseteq R, \; \von(S)=\varnothing \}.$$
\end{theorem}
\begin{proof}
First note that $B$ is the parity-check matrix of $C$. By Theorem \ref{minmallyDependentColumns}, the support of a code word in $C$ with weight $d(C)$ is the set of indices of some minimally dependent columns of $B$, say indexed by $T$ for some nonempty subset $T$ of $R$. By Theorem \ref{mindist1},  $\von(T)=\varnothing$. Then 
$$d(C)=|T| \geq  \min \{ |S|\;:\; \varnothing \neq S\subseteq R, \; \von(S)=\varnothing \}.$$
To show the equality, on the contrary suppose there is a nonempty subset $S$ of $R$ for which $d(C)> |S|$ and $\von(S)=\varnothing$. Then by Theorem \ref{mindist1}, we find $|S|$ linearly dependent columns of $B$ giving a codeword of $C$ with weight less than $d(C)$, a contradiction.
\end{proof}

\begin{example}
Consider the $(5,4,2,\frac{1}{2},\frac{2}{3})$ expander graph $G$ in Figure \ref{fig1}.  Suppose $C$ is the binary linear code whose parity-check matrix is the bi-adjacency matrix of $G$. We can verify that for any nonempty set $S\subseteq R$ with $|S|\leq 2$, we have $\von(S)\neq \varnothing$. Now for $S=\{6,8,9\}$, $\von(S)= \varnothing$. Thus by Theorem \ref{mindist}, 
$$d(C) = \min \{ |S|\;:\; \varnothing \neq S\subseteq R, \; \von(S)=\varnothing \}=|\{6,8,9\}|=3.$$
\end{example}

The preceding theorem results in a new formula for the minimum distance of expander codes.
\begin{theorem}\label{mindist2}
Suppose $G$ is a $(m,n,d,\gamma,\alpha)$ expander graph with vertex set $L \cupdot R$ such that $|L|=m$ and $|R|=n$. Then the minimum distance $d(C)$ of the expander code $C$ of $G$ is given by
$$d(C) = \min \{ |S|\;:\; \varnothing \neq S\subseteq R, \; \von(S)=\varnothing \}.$$
\end{theorem}


Using the minimum distance formula given in Theorem \ref{mindist2}, we provide a new proof of the following known result which gives a lower bound of the minimum distance of an expander code.
\begin{theorem}\label{lower bound}
Let $0<\varepsilon<\frac{1}{2}$ and $\gamma >0$ such that $\gamma n$ is a positive integer.  Suppose $G$ is a $(m,n,d,\gamma,1-\varepsilon)$ expander graph with vertex set $L \cupdot R$ such that $|L|=m$ and $|R|=n$. Then the minimum distance $d(C)$ of the expander code $C$ of $G$ has the following lower bound:
$$d(C) \geq 2(1-\varepsilon) \gamma n.$$
\end{theorem}

To prove Theorem \ref{lower bound}, we first prove the following lemmas:

\begin{lemma}\label{lemma1}
Let $0<\varepsilon<\frac{1}{2}$. Suppose $G$ is a $(m,n,d,\gamma,1-\varepsilon)$ expander graph with vertex set $L \cupdot R$ such that $|L|=m$ and $|R|=n$.  For each set $S\subseteq R$ satisfying $|S|\leq \gamma n$, we have 
$$d(1-2\varepsilon)|S| \leq |\von(S)| \leq |\n(S)|.$$

\end{lemma}
\begin{proof}
Suppose $S\subseteq R$ satisfies $|S|\leq \gamma n$. 
The second inequality follows from the fact $\von(S) \subseteq \n(S)\subseteq L$ by definition. To show the first inequality, note that there are $d|S|$ edges between vertices in $S$ and vertices in $\n(S)\subseteq L$ and each vertex in $\von(S)$ has at least one neighbor in $S$. Also each vertex in $\n(S)\setminus \von(S)$ has even number (at least $2$) of neighbors in $S$. Thus 
$$d|S|\geq |\von(S)|+2|\n(S)\setminus \von(S)|=2|\n(S)|-|\von(S)|$$
which implies
$$|\von(S)|\geq 2|\n(S)|-d|S|.$$
Since $|S|\leq \gamma n$ and $G$ is a $(m,n,d,\gamma,1-\varepsilon)$ expander graph, $|\n(S)|\geq d(1-\varepsilon)|S|$. Thus 
$$|\von(S)|\geq 2|\n(S)|-d|S|\geq 2d(1-\varepsilon)|S|-d|S|=d(1-2\varepsilon)|S|.$$



\end{proof}

\begin{lemma}\label{lemma2}
Suppose $G$ is a bipartite graph with vertex set $L \cupdot R$. Let $A$ and $B$ be nonempty disjoint subsets of $S\subseteq R$ such that $S=A\cup B$. If $\von(S)= \varnothing$, then $\von(A)= \von(B)$.
\end{lemma}
\begin{proof}
Let $\von(S)= \varnothing$. To show $\von(A)\subseteq \von(B)$, suppose $x\in \von(A)$. We claim $x\in \von(B)$. Otherwise $x\notin \von(B)$, i.e., $x$ is adjacent to an even number of vertices in $B$. Since $x\in \von(A)$, $x$ is adjacent to an odd number of vertices in $A$. Thus $x$ is adjacent to an odd number of vertices in $S=A\cup B$. Therefore $\von(S)\neq \varnothing$, a contradiction. Thus $\von(A)\subseteq \von(B)$. Similarly we can show that $\von(B)\subseteq \von(A)$.
\end{proof}

    

Using the above lemmas, we prove Theorem \ref{lower bound}.

\begin{proof}[Proof of Theorem \ref{lower bound}]
By Theorem \ref{mindist2}, consider a nonempty set $S\subseteq R$ such that $d(C) = |S|$ and $\von(S)=\varnothing$. To prove by contradiction, suppose $2(1-\varepsilon) \gamma n> d(C)= |S|$. \\

Case 1. $|S|\leq \gamma n$\\
By Lemma \ref{lemma1}, $d(1-2\varepsilon)|S| \leq |\von(S)|$. Since  $\varepsilon<\frac{1}{2}$, we have 
$$0<d(1-2\varepsilon)|S| \leq |\von(S)|,$$
which implies $\von(S)\neq \varnothing$, a contradiction. \\

Case 2. $|S|>\gamma n$\\
In this case 
$$2(1-\varepsilon) \gamma n>|S|> \gamma n.$$




Choose a nonempty subset  $T$ of $S\subseteq R$ such that $|T|= \gamma n$. Then by Lemma \ref{lemma1},
\begin{equation}\label{eq1}
    d(1-2\varepsilon)\gamma n=d(1-2\varepsilon)|T| \leq |\von(T)| \leq |\n(T)|.
\end{equation}

Note that
$$|S\setminus T|=|S|-|T|<2(1-\varepsilon) \gamma n-\gamma n=(1-2\varepsilon)\gamma n.$$
Since each vertex in $S\setminus T$ has $d$ neighbors in $L$, by Lemma \ref{lemma1},
\begin{equation}\label{eq2}
    |\von(S\setminus T)|\leq |\n(S\setminus T)|\leq d|S\setminus T|<d(1-2\varepsilon)\gamma n.
\end{equation}

Combining (\ref{eq1}) and (\ref{eq2}),  we have
$$|\von(S\setminus T)|< d(1-2\varepsilon)\gamma n\leq |\von(T)|,$$
which implies $\von(S\setminus T)\neq \von(T)$. Since $S=S\cup (S\setminus T)$ and $\von(S)=\varnothing$, by Lemma \ref{lemma2}, we have $\von(S\setminus T)= \von(T)$, a contradiction.



\end{proof}

\begin{obs}
If we like to find the minimum distance $d(C)$ of the expander code $C$ of a $(m,n,d,\gamma,1-\varepsilon)$ expander graph $G$ with vertex set $L\cupdot R$ by brute force using Theorem \ref{mindist2}, then we need to consider all possible subset $S\subseteq R$ such that $\von(S)=\varnothing$. But because of Theorem \ref{lower bound}, we need to look at only  $S\subseteq R$ satisfying $|S|>2(1-\varepsilon) \gamma n$.
\end{obs}

\begin{example}
Consider the expander code $C(G)$ of the $(5,4,2,\frac{1}{2},\frac{2}{3})$ expander graph $G$ in Figure \ref{fig1}. Note that $1-\varepsilon=\frac{2}{3}$. By Theorem \ref{lower bound}, we need to look at only  $S\subseteq R$ satisfying $|S|>2(1-\varepsilon) \gamma n=\frac{8}{3}$. So we look at nonempty sets $S\subseteq R$ satisfying $|S|\geq 3$ and verify whether  $\von(S)=\varnothing$. For $S=\{6,8,9\}$, $\von(S)= \varnothing$. Thus by Theorem \ref{mindist2}, 
$$d(C(G)) = \min \{ |S|\;:\; \varnothing \neq S\subseteq R, \; \von(S)=\varnothing \}=|\{6,8,9\}|=3.$$

\end{example}

\vspace{12pt}
\noindent {\bf Acknowledgments}\\
The author would like to thank his colleague Dr. Bahattin Yildiz for his valuable suggestions.


\begin{thebibliography}{99}




\bibitem{alon} N. Alon, J. Bruck, J. Naor, M. Naor, and R. Roth, Construction of asymptotically good low-rate error-correcting codes through pseudo-random graphs, \emph{IEEE Transactions on Information Theory}, 38:509–516, 1992.


\bibitem{Capalbo} M. R. Capalbo, O. Reingold, S. Vadhan, and A. Wigderson, Randomness conductors and constant-degree lossless expanders, \emph{Proceedings of the ACM Symposium on Theory of Computing (STOC)}, pages 659–668, 2002.

\bibitem{Mallik} Sudipta Mallik and Bahattin Yildiz, {\it Isodual and Self-dual Codes from Graphs} (under review).

\bibitem{roth} Ron M. Roth, \emph{Introduction to Coding Theory},  Cambridge University Press 2006.

\bibitem{Spielman} M. Sipser and D. Spielman, Expander codes, \emph{IEEE Transactions on Information Theory}, 42(6):1710–1722, 1996.

\bibitem{Tanner} M. Tanner, {\it A recursive approach to Low-complexity codes}, IEEE Trans. Inform. Theory, vol. IT-27, pp. 533–547, Sept. 1981.

\bibitem{Zemor} G. Zemor, On expander codes, \emph{IEEE Transactions on Information Theory}, 47(2):835–837, 2001.

\end{thebibliography}
\end{document}